\documentclass[a4paper]{article}
\usepackage[all]{xy}
\usepackage{amsmath,amssymb,amsthm,enumerate,graphicx,color,arydshln}
\usepackage[hidelinks]{hyperref}

\usepackage{aliascnt}

\newcommand{\N}{\mathbb N}
\newcommand{\Z}{\mathbb Z}

\newcommand{\C}{\mathbb C}

\newcommand{\gen}[1]{{\langle #1 \rangle}}
\newcommand{\inv}[1]{{#1^{-1}}}

\newcommand{\GL}{\mathrm{GL}}

\newcommand{\PGL}{\mathrm{PGL}}

\newcommand{\Irr}{\mathrm{Irr}}
\newcommand{\Cl}{\mathrm{Cl}}

\newcommand{\res}[2]{{#1}\hspace{-.75ex}\downharpoonright_{#2}}
\newcommand{\ind}[2]{{#1}\hspace{-.75ex}\upharpoonright^{#2}}
\newcommand{\IG}[3]{{#1}_{#2,#3}}
\renewcommand{\inf}[2]{{#1\!~\!}^{\flat_{#2}}}

\newcommand{\cg}[3]{{#2}\hspace{-.50ex}\mid_{#1}^{#3}}

\newcommand{\mynewtheorem}[2]{
  \newaliascnt{#1}{dummy}
  \newtheorem{#1}[#1]{#2}
  \aliascntresetthe{#1}
  \expandafter\def\csname #1autorefname\endcsname{#2}
}
\newtheorem*{thm*}{Theorem}

\theoremstyle{plain}
  \mynewtheorem{thm}{Theorem}
  \mynewtheorem{pps}{Proposition}
  \mynewtheorem{cor}{Corollary}
  \mynewtheorem{lem}{Lemma}
  \mynewtheorem{cjt}{Conjecture}
\theoremstyle{definition}
  \mynewtheorem{dfn}{Definition}
  \mynewtheorem{es}{Example}
\theoremstyle{remark}
  \mynewtheorem{oss}{Remark}
  \mynewtheorem{pbm}{Problem}
  
\title{A note on the projective representations of $p$-solvable and $\pi$-separable groups}
\author{Mariagrazia Bianchi and Nicola Sambonet}
\begin{document}

\maketitle

\begin{abstract}
Following Gluck and Wolf we complete the Itô--Michler's Theorem for the
projective representations of a $p$-solvable or $\pi$-separable group, and then we relate the projective irreducible modules of such a group with those of its Sylow $p$-subroups and Hall $\pi$-subgroups.
\end{abstract}

\section{Introduction}
In representation theory of finite groups, a main task is to describe how the $p$-sugroups control the $p$-part of the degrees of the irreducible modules.
In this direction, the most important results are the Itô--Michler's Teorem, the McKay's Conjecture, and the Brauer's Height Zero Conjecture, both conjectures are now proved \cite{CabanesSpath,MalleNavarroSchaefferFryTiep,Michler}.
Among the three results, the McKay's Conjecture is of different nature, counting the irreducible modules of $p'$-degree.
Indeed the Height Zero Conjecture, although it is about the interaction of ordinary and modular representations, is closer to the Itô--Michler's Theorem, and in fact its proof relies on a projective version of the theorem, proved by D.\!~Gluck and T.\!~R.\!~Wolf for solvable and $p$-solvable groups, and by G.\!~Navarro and P.\!~H.\!~Tiep in full generality \cite{GluckWolf,GluckWolfpSol,Wolf,NavarroTiep}.
Hereby $p$ is a prime,  $G$ is a finite group, $P$ is a Sylow $p$-subgroup of $G$, and $c$ is a coclass in the Schur multiplier $H^2G$.
Thus our starting point are the following results.
% For the purpose of this paper, we just mention that the ordinary irreducible modules are particioned according to some \emph{blocks}, whose definition relies on the modular representations, each block determines a conugacy class of $p$-subgroups, and a representative of this class is called a \emph{defect subgroup} of the block.
\begin{thm*}[Itô--Michler]\label{Satz:Itô Michler}
The prime $p$ does not divide the irreducible degrees of $G$ if and only if $P$ is normal abelian.
\end{thm*}
% \begin{thm*}[McKay's Conjecture]\label{Satz:McKay conjecture}
% The number of isomorphism classes of irreducible modules having degree prime to $p$ is the same in $G$ and in $N_G(P)$.
% \end{thm*}
% \begin{thm*}[Brauer's Height Zero Conjecture]\label{Satz:height zero conjecture}
% If $B$ is a $p$-block of $G$ with defect subgroup $D$, then the $p$-part of the degree of all the irreducible modules in $B$ is the $p$-part of $|G:D|$ if and only if $D$ is abelian.
% \end{thm*}
\begin{thm*}[Gluck--Wolf, Navarro--Tiep]\label{Satz: Itô Michler projective}
If the prime $p$ does not divide the irreducible $c$-degrees of $G$, then $P$ is abelian.
\end{thm*}
Our initial motivation was to complete this result by listing equivalent conditions, and in turn for $p$-solvable groups this task is achieved easily by following Gluck and Wolf (\autoref{Satz: Itô Michler projective p-sol}).
\begin{thm}\label{Intro:proj IM p-sol}
Let $G$ be $p$-solvable. The prime $p$ does not divide the irreducible $c$-degrees of $G$ if and only if $P$ is abelian, the restriction of $c$ to $P$ is trivial, and all the irreducible $c$-modules of $O_{p'}(G)$ are $P$-invariant.
Moreover, the last condition is equivalent to say that all the $c$-regular conjugacy classes of $O_{p'}(G)$ are $P$-invariant.
\end{thm}
As the condition on $O_{p'}(G)$ suggests, this result cannot be generalized to all finite groups,  and readily the alternating group $A_5$ is a counterexample.
On the other hand, this is a true generalization of the Itô--Michler's Theorem for $p$-solvable groups (\autoref{Satz: pSol all classes are c-regular}), and it extends to the $\pi$-separable groups for any set $\pi$ of primes as follows (\autoref{Satz: pi-sep abelian factors and Sylow subgroups}).
\begin{thm}\label{Intro:proj IM pi-sep}
Let $G$ be $\pi$-separable.
Then every irreducible $c$-module of $G$ has $\pi'$-degree if and only if $G$ is $p$-solvable and satisfies the conditions of \autoref{Intro:proj IM p-sol} for all primes $p$ in $\pi$.
In this case every $\pi$-factor of $G$ is abelian, and $c$ restricts trivially to the Hall $\pi$-subgroups of $G$.
\end{thm}
% it is natural to ask when the $p$-part of some irreducible degree of a group $G$ is an irreducible degree of some $p$-Sylow subgroup $P$.
%
% Let $G$ be a finite group, and $P$ be a $p$-Sylow subgroup of $G$ for some prime number $p$.
% Suppose that $V$ is an irreducible module of $G$, over the complex field, such that $p$ divides the degree of $V$.
% We may expect that certain irreducible module $V_p$ of $P$ controls the $p$-part of $V$, in sense that $V_p$ is somehow related to $V$, and $\dim V_p=(\dim V)_p=p^a$ for $\dim V=p^am$ with $(m,p)=1$.
% Of course, it is enough to look at few examples to percieve that a phenomena of this kind does not occur in the general case, but the situation is different if we assume that the group $G$ is $p$-solvable and that the module $V$ is primitive. In this case,  by an argument, we have the following result.
In these variations on the Itô--Michler's Theorem, to consider all the irreducible $c$-modules at once is fundamental since in general we do not know how to separate the $\pi$-and $\pi'$-part of a single $c$-module $V$ in any reasonable way.
Remarkably this task is possible when $G$ is $\pi$-separable and $V$ is primitive, essentially by an argument of Schur (\autoref{Satz: n-decomposition}).
% \begin{thm}
% Let $V$ be a primitive irreducible module and $P$ be a Sylow $p$-subgroup of a $p$-solvable group $G$.
% Then $V=V_p\otimes V_{p'}$ where $(\dim V)_p=\dim V_p$ and the restricted module $\res{V_p}{P}$ is irreducible.
% \end{thm}
% The same argument generalizes to any $\pi$-separable group $G$ and its Hall $\pi$-subgroups, where $\pi$ is any set of prime numbers.
% Moreover, it generalizes to the $c$-modules for any coclass $c$ in the Schur multiplier $H^2G$.
Since $H^2G$ is a finite abelian group, we can decompose
a coclass as a product $c=c_\pi c_{\pi'}$ where the order of $c_\pi$ and
$c_{\pi'}$ are $\pi$-and $\pi'$-numbers, and we  have the following result.
\begin{thm}
Let $G$ be $\pi$-separable, with Hall $\pi$-and $\pi'$-subgroups $H$ and $H'$.
Let $V$ be a primitive irreducible $c$-module of $G$.
Then $V=V_{\pi}\otimes V_{\pi'}$ for some $c_\pi$-and $c_{\pi'}$-modules $V_\pi$
and $V_{\pi'}$, so that $\res{V_\pi}{H}$ and $\res{V_{\pi'}}{H'}$ are
irreducible.
In particular $(\dim V)_\pi=\dim V_\pi$ and
$(\dim V)_{\pi'}=\dim V_{\pi'}$.
Moreover, since
$\res{c_\pi}{H'}=1$ and $\res{c_{\pi'}}{H}=1$, then $\res{V_{\pi}}{H'}$ and $\res{V_{\pi'}}{H}$ are ordinary modules.
\end{thm}
The primitivity can be relaxed by using Clifford's Theory, and indeed we prove the above result in this form (\autoref{Satz: pi-decomposition}).
In particular, we find also the following application (\autoref{Satz: pi-decomposition degree}).
\begin{cor}
Let $G$ be $\pi$-separable.
An irreducible $c$-module $V$ has $\pi'$-degree if and only if it is induced from a $c$-module $V'$ of some subgroup $J$, so that $J$ contains a Hall $\pi$-subgroup of $G$ and $V'$ restricts irreducibly to the Hall $\pi'$-subgroups of $J$.
\end{cor}
% In representation theory of finite groups, the  Itô--Michler Theorem relates the arithmetic of the irreducible representations to the algebraic structure of the group \cite{Isaacs,HuppertChar,Navarro}.
The remaining of this paper is divided into two sections, in \autoref{section:background} we provide some background, and in \autoref{section:results} we prove our results.

\section{Background}\label{section:background}
There are two ways to introduce this subject, starting from either ordinary representation theory or cohomology, and some standard references in the field are \cite{Benson,Brown,HuppertChar,Isaacs,Karpilovsky2}.
A \emph{projective representation} of $G$ is a homomorphism $\psi:G\to\PGL(V)$, for some complex vector space $V$ of finite dimension, which is the \emph{degree} of $\psi$.
A \emph{section} of $\psi$ is a lifting $\varphi:G\to\GL(V)$, and this determines a \emph{cocycle} $\alpha$ by the relation $x\varphi\!~y\varphi=\alpha(x,y)\!~xy\varphi$.
The cocycle associated to another section $\varphi'$ is the product $\alpha'=\alpha\cdot\delta\zeta$, where $\delta\zeta$ is a \emph{coboundary}.
The sets of the cocycles and the coboundaries are abelian groups by pointwise multiplication, denoted by $Z^2G$ and $B^2G$.
Therefore, $\psi$ determines a \emph{coclass} $c$ in the \emph{Schur multiplier} $H^2G=Z^2G/B^2G$, and we say that $\psi$ is a \emph{$c$-representation} and $V$ is a \emph{$c$-module}.

Like in the ordinary case, a $c$-module $V$ is \emph{irreducible} if its only submodules are the trivial ones, that is when $\varphi$ is not similar to some upper triangular form.
We denote by $\Irr(G|c)$ the similarity classes of irreducible $c$-modules, which we still identify with a set of representatives $\{V_1,\ldots,V_r\}$, and we say that $V$ lies in $\Irr(G|c)$.
The direct sum $U\oplus V$ of $c$-modules is also a $c$-module, and every $c$-module is the sum of irreducible $c$-modules.
On the other hand, if $U$ is a $a$-module and $V$ is a $b$-module, for $a,b\in H^2G$, then their tensor product is an $ab$-module.
Precisely, if $\varphi:G\to\GL(U)$ and $\upsilon:G\to\GL(V)$ have cocycles $\alpha$ and $\beta$, where $a=[\alpha]$ and $b=[\beta]$, then the diagonal action $\omega=\varphi\otimes\upsilon:G\to\GL(U\otimes V)$ has cocycle $ab=[\alpha\beta]$.

Also for the $c$-representations we have an associative algebra, the \emph{twisted group algebra} $\C^cG$, with basis a copy $G\sigma$ of $G$, where multiplication is defined on the basis by $g\sigma\!~h\sigma=\alpha(g,h)\!~gh\sigma$ and extended by linearity.
The algebra is associative precisely because $\alpha$ is a cocycle, on the other hand, by setting $g\sigma'=\zeta(g)\!~g\sigma$ we have a basis $G\sigma'$ associated to $\alpha'=\alpha\cdot\delta\zeta$, for this reason we avoid the more common notation $\C^\alpha G$.
% The composition of the map $\sigma$ with an algebra homomorphism $\varphi:\C^cG\to\End(V)$ is the section of a $c$-representation of $G$.
The algebra $\C^cG$ is semisimple, so we have Wedderburn's Theorem and the degree formula $\sum_i(\dim V_i)^2=|G|$.
The class sums of the $c$-regular classes $\Cl(G|c)$ form a basis of the center $Z(\C^cG)$, thus $|\Irr(G|c)|=|\Cl(G|c)|$, and the number of irreducible $c$-modules is the number of $c$-regular conjugacy classes.
An element $x$ of $G$ is \emph{$c$-regular} if the class sum $\inv{|G_x|}\sum_g x\sigma^{g\sigma}$ is non--zero.
Since $x\sigma^{g\sigma}=\cg{\alpha}{x}{g}x^g\sigma$ for $\cg{\alpha}{g}{y}=\alpha(x,g)\inv{\alpha(g,x^g)}$, the class sum depends on the representative $x$, still conjugate elements have collinear sums and $c$-regularity is a property of $x^G$.
% , this is equivalent to say that $\cg{\alpha}{x}{z}=1$ for every element $z$ in the centralizer $G_x$, and it is
On the other hand, if $c\neq 1$ there are no $c$-modules of degree one, as we have the following result (see \cite{Karpilovsky2}).
\begin{thm}\label{Satz: o(c), dim V and oG}
  The order $o(c)$ divides the degree of every $c$-module of $G$, the degree of every  irreducible $c$-module divides $|G|$, and $o(c)^2$ divides $|G|$.
\end{thm}

% Also for the projective representations there is the notion of induction.
For $H\leq G$ by restriction we have a coclass $\res{c}{H}$, thus $\C^cH\leq\C^cG$.
and every $\res{c}{H}$-module $U$  induces a $c$-module $\ind{U}{^cG}=U\otimes_{\C^cH}\C^cG$.
% Precisely, if $G\sigma$ is a group basis of $\C^cG$ with cocycle $\alpha$, and $\varphi$ is a section of the representation of $H$ over $U$ with cocycle $\res{\alpha}{H}$, then the action of $G$ over $\ind{U}{^cG}$ is described by the equations $(u\otimes w)g=u\otimes (w\!~g\sigma)$ and $(u\!~h\varphi)\otimes w=u\otimes(w\!~h\sigma)$,
% for all $u\in U$, $w\in\C^cG$, $g\in G$, and $h\in H$.
% Also, by choosing a basis $B=\{b_1,\ldots,b_n\}$ of $U$ and a transversal $T=\{t_1,\ldots,t_m\}$ for $H$ in $G$, then $B\otimes T\sigma=\{\!~b_i\otimes t_j\sigma\!~|\!~i,j\!~\}$ is a basis of $\ind{U}{^cG}$.
The basic properties of induction are transitivity, Frobenius reciprocity, and Mackey's decomposition (see \cite{Benson,Karpilovsky2}).
\begin{thm}\label{Satz: properties of induction}
Let $H,K,L\leq G$ with $H\leq K$, let $T$ be a double transversal for $H$ and $L$ in $G$, let $V$ be a $b$-module and $U$ a $\res{c}{H}$-module for $b,c\in H^2G$.
\begin{enumerate}[i)]
\item $\ind{U}{^cG}=\ind{\ind{U}{^cK}}{^cG}$;
\item
$V\otimes(\ind{U}{^cG})=\ind{(\res{V}{H}\otimes U)}{^{bc}G}$;
\item
$\displaystyle\res{\ind{U}{^cG}}{L}=\sum_{t\in T}\ind{\res{U^t}{H^t\cap L}}{^cL}$.
\end{enumerate}
\end{thm}
The projective representations relate the representations of a group to those of its normal subgroups.
This connection origins with the work of Schur and flourishes with Clifford's theory, and the contribution of Mackey \cite{HuppertChar,Isaacs,Karpilovsky2}.
First, when $N\trianglelefteq G$ then we have the restriction $\res{}{N}:H^2G\to H^2N$ but also the \emph{inflation} $\inf{}{}:H^2(G/N)\to H^2G$,
which regards a cocycle of $G/N$ as one of $G$ by setting $\alpha(g,h)=\alpha(Ng,Nh)$.
The inflation is not necessarily injective over the coclasses,
% unless $N$ has a complement $H$ and we have an isomorphism $H^2(G/N)\simeq H^2 H$.
% In general
still $\inf{H^2(G/N)}{}\leq\ker\res{}{N}$, and the condition $c\in\inf{H^2(G/N)}{}$ is equivalent to have $c=[\alpha]$ for some cocycle satisfying $\res{\alpha}{N}=1$ and $\cg{\alpha}{x}{g}=1$ for all $x\in N$ and $g\in G$.
Second, for $b\in H^2(G/N)$, then a $b$-module $W$ is a $\inf{b}{}$-module simply by composition $G\to G/N\to\GL(W)$, and we may drop the symbol $\inf{}{}$ without ambiguity.
Moreover, since $\res{b}{N}=1$, then $\res{W}{N}\simeq\C^n$ is a trivial module of degree $n=\dim(W)$.
It follows that $\res{(Y\otimes W)}{N}\simeq(\res{Y}{N})^n$ for every $c$-module $Y$ of $G$ and, in particular, the irreducible components of $\res{(Y\otimes W)}{N}$ are the same of $\res{Y}{N}$.
In fact $Y\otimes W$ is irreducible precisely when $V=\res{Y}{N}$ and $W$ are irreducible, and in this case $V$ is $G$-invariant.
To describe a centralizer $\IG{G}{c}{V}$ for the action of $G$ over $\Irr(N|c)$ is more delicate than in the ordinary case.
% \begin{dfn}
Let $V$ be an irreducible $\res{c}{N}$-module, so that $\varphi:N\to\GL(V)$ is associated with the cocycle $\res{\alpha}{N}$, then the \emph{inertia subgroup} of $V$ with respect to $c$ is
\[\IG{G}{c}{V}=\left\{\!~g\in G\!~|\!~\exists\!~ y\in\GL(V)\!~:\!~x\varphi^y=\cg{\alpha}{x}{g}x^g\varphi\!~,\!~\forall\!~x\in N\!~\right\}\]
and $V$ is \emph{$\C^cG$-invariant}  if $\IG{G}{c}{V}=G$.
In fact $\IG{G}{c}{V}$ depends on the coclass $c$ rather than the cocycle $\alpha$ and, in addition, $\IG{G}{cb}{V}=\IG{G}{c}{V}$ for all $b\in H^2(G/N)$.
We denote by $\Irr(G|c,V)$ the similarity classes of irreducible $c$-modules whose restriction to $N$ has $V$ among the constituents, so that $X\in\Irr(G|c,V)$ if and only if $\res{X}{N}\simeq V\oplus Z$ for some $\res{c}{N}$-module $Z$.
In this regard the two main theorems of Clifford's Theory are the following, with Part II due to Schur. They can be written formally as
\begin{equation}\label{eq:Schur Clifford}
\Irr(G|c,V)=\ind{(Y\otimes\Irr(\IG{G}{c}{V}/N|\!~b\!~))}{^cG} \ \ \ \ ,\ \ \ \ V'=\ind{(Y\otimes W)}{^cG}
\end{equation}
and they can be read: every irreducible $\res{c}{N}$-module extends to its inertia subgroup and then it induces irreducibly to $G$; conversely, every irreducible $c$-module whose restriction to $N$ has $V$ among its constituent is of the form $\ind{(Y\otimes W)}{^cG}$ for some irreducible $b$-module $W$ of $\IG{G}{c}{V}/N$.
\begin{thm}\label{Satz: Clifford}
Let $V\in\Irr(N|c)$ and $J=\IG{G}{c}{V}$, for some $N\trianglelefteq G$ and $c\in H^2G$.
\begin{enumerate}[I.]
  \item 
If $V'\in\Irr(G|c,V)$, then $\res{V'}{N}$ is the sum of the $G$-conjugates of $V$ occurring all with the same multiplicity.
Moreover, the induction $\ind{}{^cG}$ defines a bijection $\Irr(J|c,V)\to\Irr(G|c,V)$.
  \item There is a unique coclass $\hat{c}\in H^2G$ such that the module $V$ extends to some module $Y\in\Irr(J|\!~\hat{c}\!~)$.
  Moreover, $\hat{c}\cdot\inf{b}{}=c$ for some $b\in H^2(J/N)$, and the map $W\mapsto Y\otimes W$ defines a bijection $\Irr(J/N|\!~b\!~)\to\Irr(J|c,V)$.
\end{enumerate}
\end{thm}
Next we recall some facts on $p$-solvable and $\pi$-separable groups.
We denote by $\Pi(n)$ the set of the positive prime divisors of an integer $n$, and we write $\Pi(X)=\bigcup_{n\in X}\Pi(n)$ for all $X\subseteq\Z$.
Thus $\Pi(\N)$ is our notation for the set of all the positive prime numbers.
For $\pi\subseteq\Pi(\N)$, we write $\pi'$ for the complement of $\pi$ in $\Pi(\N)$.
For $n\in\N$ we have $n=n_{\pi}n_{\pi'}$ where $n_{\pi}$ and $n_{\pi'}$ are the $\pi$-part and the $\pi'$-part of $n$, satisfying $\Pi(n_{\pi})\subseteq \pi$ and $\Pi(n_{\pi'})\subseteq \pi'$ respectively.
When $G$ is a finite group and $g$ is one of its elements, we write $\Pi(G)=\Pi(|G|)$ and $\Pi(g)=\Pi(o(g))$.
If $\Pi(G)=\{p_1,\ldots,p_k\}$, we write uniquely $g=g_{p_1}g_{p_2}\ldots g_{p_k}$ where $g_{p_i}\in\gen{g}$ and $\Pi(g_{p_i})=\{p_i\}$.
The group $O_{\pi}(G)$ is the maximal normal subgroup $N$ of $G$ satisfying $\Pi(N)\subseteq\pi$.
Inductively, given $\pi_1,\ldots,\pi_{k}\subseteq\Pi(\N)$, the group $O_{\pi_1\pi_2\ldots\pi_k}(G)$ is the correspondent of $O_{\pi_k}(G/O_{\pi_1\ldots\pi_{k-1}(G)})$ in $G$.
The group $G$ is \emph{$\pi$-separable} if it has a normal series $N_\ast:1\leq N_1\leq\ldots\leq N_{l-1}\leq G$ such that $\pi(N_{i}/N_{i-1})\subseteq\pi$ or $\pi(N_{i}/N_{i-1})\subseteq\pi'$ for all $i=1,\ldots,l$.
Every $\pi$-separable group has a characteristic normal $\pi$-series \[O_{\pi\ast}(G):1\leq O_{\pi}(G)\leq O_{\pi\pi'}(G)\leq O_{\pi\pi'\pi}(G)\leq\ldots\leq G\] called \emph{the $\pi$-series}, there is a Hall $\pi$-subgroup in $G$, and every $\pi$-subgroup is contained in some Hall $\pi$-subgroup of $G$.
There is possibly more than one conjugacy class of Hall $\pi$-subgroups.
Clearly $\pi'$-separability is the same condition of $\pi$-separability, and so a $\pi$-separable groups has the $\pi'$-series $O_{\pi'\ast}(G)$ and some Hall $\pi'$-subgroup in $G$.
For $\pi=\{p\}$ the definition of $p$-separable group coincides with that of $p$-solvable group, we have the series \[O_{p\ast}(G):1\leq O_{p}(G)\leq O_{pp'}\leq O_{pp'p}\leq\ldots\leq G\!~,\] \[O_{p'\ast}(G):1\leq O_{p'}(G)\leq O_{p'p}\leq O_{p'pp'}\leq\ldots\leq G\!~.\]
The latter, that is the $p'$-series, satisfies a fundamental property discovered by P.\!~Hall and G.\!~Higman
\cite[Lemma 1.2.3]{HallHigman}.
\begin{lem}\label{Satz: Hall Higman}
Let $G$ be $p$-solvable.
If $O_{p'}(G)=1$, then $C_G(O_{p}(G))\leq O_{p}(G)$.
In general, we have that %$C_G(O_{p'p}(G)/O_{p'}(G))\leq O_{p'p}(G)$, and as well that
$C_G(O_{(p'p)^{k+1}}(G)/O_{(p'p)^kp'})\leq O_{(p'p)^{k+1}}(G)$.
% for all $k\in\N$.
\end{lem}
For a coclass $c$ in $H^2G$ we write $\Pi(c)=\Pi(o(c))$, for a $c$-module $V$ we write $\Pi(V)$ in place of $\Pi(\dim V)$, and thus $\Pi(\Irr(G|c))$ denotes the set of prime divisors of the irreducible degrees.
We have a fundamental result that is actually true for the higher cohomology groups (see \cite[Theorem 10.3]{Brown}).
\begin{thm}\label{Satz: H2Gp and H2H}
  If $H$ and $(H^2G)_{\pi}$ are Hall $\pi$-subgroups of $G$ and $H^2G$, then the restriction $\res{}{H}:(H^2G)_{H}\to H^2H$ is injective.
\end{thm}
In particular
%, since $H^2H$ is an abelian $\pi$-group, for every coclass $c$ in $H^2G$ we have that
$\Pi(c)\subseteq\pi'$ precisely when $c$ restricts trivially to $H$ and, combining this fact with \autoref{Satz: o(c), dim V and oG}, we see that the primes dividing the order of a coclass $c$ form a subset of the primes dividing the degrees of the irreducible $c$-modules.
\begin{cor}\label{Satz: pic piIrrGc piG}
 We have  $\Pi(c)\subseteq\Pi(V)\subseteq\Pi(\Irr(G|c))\subseteq\Pi(G)$ for every irreducible $c$-module $V$ of $G$.
  If $H$ is a Hall $\pi$-subgroup of $G$, then the statements $\Pi(c)\subseteq\pi'$, $\res{c}{H}=1$, and $\dim(X)=1$ for some $X\in\Irr(H|c)$, are equivalent.
\end{cor}

\section{Results}\label{section:results}
Decomposing an irreducible $c$-module $W$ of $G$ over a normal subgroup $N$ gives some information about the prime divisors of its degree.
Indeed, a direct consequence of Clifford's theory is that, if a prime $p$ does not divide the degree of $W$, then it does not divide the degree of the irreducible constituents of $\res{W}{N}$ and, moreover, every such constituent is invariant under the action of some Sylow $p$-subgroup $P$ of $G$.
Another consequence is that an irreducible $c$-module $V$ of $N$ that is invariant under a coprime action is extendable.
A more precise statement can be phrased as follows.
\begin{lem}\label{Satz: Clifford prime and coprime}
With the notation of \autoref{Satz: Clifford}, we have the following statements.
\begin{enumerate}[i)]
\item $\Pi(V')=\Pi(V)\cup\Pi(W)\cup\Pi(|G:J|)$.
\item $\Pi(\Irr(G|c,V))=\Pi(V)\cup\Pi(\Irr(J/N|b))\cup\Pi(|G:J|)$.
\item $\Pi(\Irr(N|c))\subseteq\Pi(N)\cap\Pi(G|c)$.
\item If $\Pi(N)\cap\Pi(J/N)=\varnothing$ and $\Pi(c)\subseteq\Pi(N)$ then $b=1$, so that $Y\in\Irr(J|c)$ and $W\in\Irr(J/N)$.
\item
If $G$ is $\pi$-separable and $\Pi(V')\subseteq\pi'$, for some set of prime $\pi$, then $J$ contains a Hall $\pi$-subgroup of $G$.
\end{enumerate}
\end{lem}
\begin{proof}
$i)$ Simply $\dim(V')=\dim(V)\cdot\dim(W)\cdot|G:J|$.
$ii)$ It follows since $\Irr(G|c,V)=\ind{(Y\otimes\Irr(J/N|b))}{^cG}$.
$iii)$ It follows since $\Irr(G|c)$ is the union of the sets $\Irr(G|c,V)$ for $V\in\Irr(N|c)$.
$iv)$ Let $\pi=\Pi(N)$.
By the Schur--Zassenhaus Theorem, there is a Hall $\pi'$-subgroup $H'$ in $J$.
By \autoref{Satz: H2Gp and H2H}, the composition of $\inf{}{}:H^2(J/N)\to(H^2G)_{\pi'}$ with $\res{}{H'}:(H^2G)_{\pi'}\to H^2H'$ is injective.
We have that $\res{\inf{b}{}}{H'}=\res{(c\cdot\inv{\hat{c}})}{H'}$, since $\Pi(c)\subseteq\pi$ and, by \autoref{Satz: pic piIrrGc piG}, $\Pi(\hat{c})\subseteq\Pi(Y)=\Pi(V)\subseteq\pi$, then $\res{c}{H'}=\res{\hat{c}}{H'}=1$, proving that $\res{\inf{b}{}}{H'}=1$,
that is $b=1$ since $J/N\simeq H'$.
$v)$ The condition that $\Pi(|G:J|)\subseteq\pi'$ is the same of having a Hall $\pi$-subgroup of $G$ contained in $J$.
\end{proof}
The first consequence of this lemma is that the primes dividing the irreducible degrees control the abelianity of the $\pi$-factors of $G$, that is to say the quotient groups $M/N$ where $N,M\trianglelefteq G$ with $N\leq M$ and $\Pi(M/N)\subseteq\pi$.
\begin{cor}\label{Satz: M/N pi-abelian}
Let $M/N$ be a $\pi$-factor of $G$.
If there is a coclass $c$ in $H^2G$, and an irreducible $\res{c}{N}$-module $V$ such that $\Pi(\Irr(G|c,V))\subseteq\pi'$, then $M/N$ is abelian.
\end{cor}
\begin{proof}
By $ii)$ we have that $\Pi(\Irr(J/N|b))\cup\Pi(|G:J|)\subseteq\pi'$
for $J=\IG{G}{c}{V}$.
Since $N\leq M\cap J$ and $|M:M\cap J|=|MJ:J|$, then \[\Pi(|M:M\cap J|)\subseteq\Pi(M/N)\cap\Pi(|G:J|)=\varnothing\] and so $M\leq J$.
Since $M\trianglelefteq G$, then $M/N\trianglelefteq J/N$ and by $iii)$ we have that \[\Pi(\Irr(M/N|b))\subseteq\Pi(M/N)\cap\Pi(\Irr(J/N|b))=\varnothing\!~.\]
Therefore $\res{b}{M/N}=1$ and $\Pi(\Irr(M/N))=\varnothing$, that occurs exactly when $M/N$ is abelian.
\end{proof}
Thus, we have a generalization of the Itô--Michler Theorem for a particular class of $\pi$-separable groups.
\begin{thm}\label{Satz: Itô Michler projective pi-sep}
Let $G=O_{\pi'\pi\pi'}(G)$, and $H$ be a Hall $\pi$-subgroup of $G$.
Then the following conditions are equivalent:
\begin{enumerate}[i)]
  \item $\Pi(\Irr(G|c))\subseteq\pi'$
  \item $H$ is abelian, $\res{c}{H}=1$, and $H\leq\IG{G}{c}{V}$ for every $V\in\Irr(O_{\pi'}(G)|c)$
  \item $H$ is abelian, $\res{c}{H}=1$, and
  $H\leq G_{\tau}$ for every $\tau\in\Cl(O_{\pi'}(G)|c)$
\end{enumerate}
\end{thm}
\begin{proof}
Let $N=O_{\pi'}(G)$.
Since $G=O_{\pi'\pi\pi'}(G)$, then $O_{\pi'\pi}(G)=NH$.
$i)\Rightarrow ii)$
\autoref{Satz: o(c), dim V and oG} yields $\res{c}{H}=1$.
By \autoref{Satz: Clifford prime and coprime}, for every $V\in\Irr(N|c)$ there is a Hall $\pi$-subgroup $\tilde H$ of $G$ contained in $J=G_{c,V}$. Since $NH=O_{\pi'\pi}(G)$, then $N\tilde H=NH$ so that $H\leq J$.
Moreover, $H$ is isomorphic to the $\pi$-factor $NH/N$ which is abelian by \autoref{Satz: M/N pi-abelian}.
$ii)\Rightarrow i)$
For any $V'\in\Irr(G|c)$, we have to show that $\Pi(V')\subseteq\pi'$.
We apply twice \autoref{Satz: Clifford}.
First, we choose a constituent $V_1$ of $\res{V'}{N}$, and denote $J_1=\IG{G}{c}{V_1}$, so that there are $Y_1\in\Irr(J_1|\hat{c})$ and $W_1\in\Irr(J_1/N|b_1)$ such that $V'=\ind{(Y_1\otimes W_1)}{^cG}$ and $\res{Y_1}{N}=V_1$.
By hypothesis $H\leq J_1$ and $NH\trianglelefteq G$, thus we can choose a constituent $V_2$ of $\res{W_1}{NH}$ and denote $J_2=\IG{J_1}{,b_1}{V_2}$, so that there are $Y_2\in\Irr(J_2|\hat{b}_1)$ and $W_2\in\Irr(J_2/NH|b_2)$ such that $W_1=\ind{(Y_2\otimes W_2)}{^bG}$ and $\res{Y_2}{NP}=V_2$.
Thus $\Pi(V')=\Pi(V_1)\cup\Pi(V_2)\cup\Pi(W_2)\cup\Pi(|G:J_2|)$.
Regarding $W_2$ as a module of $J_1$, and applying \autoref{Satz: Clifford prime and coprime} to $NH$, we have that $\res{b}{H}=1$ and so $\res{W_1}{H}$ is an ordinary module.
Since $H$ is abelian and $V_2$ is an ordinary module, then $\dim(V_2)=1$.
On the other hand, since $N$ and $G/NH$ are $\pi'$-groups, then $\pi(V_1)\cup\Pi(W_2)\cup\Pi(|G:J_2|)\subseteq\pi'$. Thus $p\notin\Pi(V')$ as desired.
$ii)\iff iii)$ Both conditions are equivalent to the statement $C_H(Z(\C^cN))=H$.
\end{proof}
The Gluck--Wolf's Theorem relies on similar ideas, when they aim to describe the structure of a counterexample (see \cite[Proposition 0]{GluckWolf}).
Also, their result extends the above characterization to all $p$-solvable groups.
\begin{thm}\label{Satz: Itô Michler projective p-sol}
Let $G$ be $p$-solvable, $P$ be a Sylow $p$-subgroup of $G$, and $c$ be a coclass in $H^2G$. Then the following conditions are equivalent:
\begin{enumerate}[i)]
  \item $p\notin\Pi(\Irr(G|c))$
  \item $P$ is abelian, $\res{c}{P}=1$, and $P\leq\IG{G}{c}{V}$ for every $V\in\Irr(O_{p'}(G)|c)$
  \item $P$ is abelian, $\res{c}{P}=1$, and
  $P\leq G_{\tau}$ for every $\tau\in\Cl(O_{p'}(G)|c)$
\end{enumerate}
\end{thm}
\begin{proof}
The statements are those of \autoref{Satz: Itô Michler projective pi-sep} with $\pi=\{p\}$.
On one hand, if $G$ is $p$-solvable and  $P$ is abelian, then $G=O_{p'pp'}(G)$ by Hall--Higman \autoref{Satz: Hall Higman}.
On the other hand, if $i)$ is satisfied, then $P$ is abelian by the Gluck--Wolf's Theorem.
\end{proof}
This is a true generalization of the Itô--Michler Theorem for $p$-Solvable groups, since for the trivial coclass the condition $iii)$ is equivalent to say that $P$ is normal abelian.
Indeed, this is true for a broader family of coclasses:
\begin{cor}\label{Satz: pSol all classes are c-regular}
Let $G$ be $p$-solvable and assume that the conjugacy classes of $O_{p'}(G)$ are $c$-regular.
Then $p$ does not divide the irreducible $c$-degrees of $G$ if and only if $G$ has a normal abelian Sylow $p$-subgroup $P$ and the restriction of $c$ to $P$ is trivial.
\end{cor}
\begin{proof}
By \autoref{Satz: Itô Michler projective p-sol}, we have that $p\notin\Pi(\Irr(G|c))$ if and only if $G$ has an abelian Sylow $p$-subgroup $P$, $\res{c}{P}=1$, and $P$ fixes $\Cl(O_{p'}(G)|c)$.
However, the additional hypothesis is that $\Cl(O_{p'}(G)|c)=\Cl(O_{p'}(G))$, and so $P$ fixes $\Cl(O_{p'}(G))$ if and only if it fixes $\Irr(O_{p'}(G))$.
Since $P$ is abelian, this is equivalent to say that $p\notin\pi(\Irr(O_{p'}(G)P))$ and, by the Itô--Michler's theorem, to say that $P\trianglelefteq O_{p'}(G)P$, that is $P\trianglelefteq G$.
\end{proof}
For $\pi$-separable groups, joining Gluck--Wolf's Theorem and \autoref{Satz: M/N pi-abelian} gives the following result, including the abelianity of all the $\pi$-factors.
% We could not prove the abelianity of the Hall $\pi$-subgroups, or find a counterexample for it.
\begin{thm}\label{Satz: pi-sep abelian factors and Sylow subgroups}
Let $G$ be $\pi$-separable and $c\in H^2G$.
Then $\Pi(\Irr(G|c))\subseteq\pi'$ if and only if $G$ is $p$-solvable and it satisfies the conditions of \autoref{Satz: Itô Michler projective p-sol} for all prime $p$ in $\pi$.
In this case every $\pi$-factor of $G$ is abelian, and $\res{c}{H}=1$ for every Hall $\pi$-subgroup $H$ of $G$.
\end{thm}
\begin{proof}
Obviously, if $G$ is $p$-solvable for every prime $p$ in $\pi$, then by Gluck--Wolf's Theorem we have that $\Pi(\Irr(G|c)\subseteq\pi'$ if and only if $G$ satisfies the conditions of \autoref{Satz: Itô Michler projective p-sol} for all $p$ in $\pi$.
On the other hand, if $\Pi(\Irr(G|c)\subseteq\pi'$, then by \autoref{Satz: M/N pi-abelian} every $\pi$-factor of $G$ is abelian, so that $G$ is $p$-solvable for every  prime $p$ in $\pi$.
\end{proof}
To prove \autoref{Satz: Itô Michler projective pi-sep} we applied \autoref{Satz: Clifford} twice, so we apply it recursively to obtain the following decomposition theorem.
\begin{thm}\label{Satz: n-decomposition}
Let $N_\ast:1=N_0\leq\ldots\leq N_{i-1}\leq N_{i}\leq\ldots\leq N_l=G$ be a normal series, and $V$ be an irreducible $c$-module of $G$.
Then we have a subgroup $J\leq G$ and a decomposition $V=\ind{(Y_1\otimes Y_2\otimes\ldots\otimes Y_l)}{^cG}$ where each $Y_i$ is an irreducible $c_i$-module of $J/J\cap N_{i-1}$, with the property that $V_i=\res{Y_i}{J\cap N_i}$ is irreducible, for some coclass $c_i$.
\end{thm}
\begin{proof}
We work by induction on $i$, when $i=0$ we take $Y_0=\C$, $W_0=V$, $J_0=G$, $M_0=1$, and $b_0=c$.
Having that $V=\ind{(Y_1\otimes\ldots\otimes Y_i\otimes W_i)}{^cG}$ with $Y_k\in\Irr(J_i/M_{k-1}|c_{k})$ for $k=1,\ldots,i$, and $W_i\in\Irr(J_i/M_i|b_i)$, we define $M_{i+1}=J_i\cap N_{i+1}$, choose an irreducible constituent $V_{i+1}$ of $\res{W_i}{M_{i+1}}$, and define $J_{i+1}=\IG{(J_i)}{b_i}{V_{i+1}}$.
By \autoref{Satz: Clifford}, there are two projective irreducible modules $Y_{i+1}\in\Irr(J_{i+1}/M_{i}|c_{i+1})$ and $W_{i+1}\in\Irr(J_{i+1}/M_{i+1}|b_{i+1})$, for some coclasses $c_{i+1}$ and $b_{i+1}$ satisfying $c_{i+1}\inf{b}{}_{i+1}=b_i$, such that $\res{Y_{i+1}}{M_{i+1}}=V_{i+1}$ and $W_i=\ind{(Y_{i+1}\otimes W_{i+1})}{^{b_i}J_i}$.
By the basic properties of the induction of \autoref{Satz: properties of induction}, we have that
$W=\ind{(\res{(Y_1\otimes\ldots\otimes Y_{i})}{J_{i+1}}\otimes Y_{i+1}\otimes W_{i+1})}{^cG}$, and so we write $Y_k$ in place of $\res{Y_k}{J_{i}}$ for $k=1,\ldots,i-1$.
When we reach $i+1=l$, then we define $J=J_{l-1}$ and  $Y_l=W_{l-1}\in\Irr(J/M_{l-1}|c_{l-1})$, to obtain the decomposition $U=\ind{(Y_1\otimes\ldots\otimes Y_l)}{^cG}$.
\end{proof}
By taking a $\pi$-series of a $\pi$-separable groups we obtain the following.
\begin{thm}\label{Satz: pi-decomposition}
Let $V$ be an irreducible $c$-module of a $\pi$-separable group $G$.
Then $V=\ind{(V_{\pi}\otimes V_{\pi'})}{^cG}$ where $V_{\pi}\in\Irr(J|c_\pi)$ and $V_{\pi'}\in\Irr(J|c_{\pi'})$ for some subgroup $J$ of $G$.
Moreover if $H$ and $H'$ are a Hall $\pi$-subgroup and a Hall $\pi'$-subgroup of $J$, then $\res{V_{\pi}}{H}$ and $\res{V_{\pi'}}{H'}$ are irreducible, whereas $\res{V_{\pi'}}{H}$ and $\res{V_{\pi}}{H'}$ are ordinary modules.
\end{thm}
\begin{proof}
Let $N_\ast:1=N_0\leq N_1\leq\ldots\leq N_{2l}\leq G$ be a normal series of $G$ such that $\Pi(N_{2i-1}/N_{2i-2})\subseteq\pi$ and $\Pi(N_{2i}/N_{2i-1})\subseteq\pi'$, for instance we can take the $\pi$-series $N_\ast=O_{\pi\ast}(G)$.
By \autoref{Satz: n-decomposition}, we have a subgroup $J$ of $G$ and a decomposition $V=\ind{(Y_1\otimes\ldots\otimes Y_{2l})}{^cG}$  where each $Y_i$ is a projective irreducible $c_i$-module of $J/M_{i-1}$ such that $V_i=\res{Y_i}{M_i}$ is irreducible, for $M_\ast=J\cap N_\ast$.
Let $V_{\pi}=Y_1\otimes Y_3\otimes\ldots\otimes Y_{2l-1}$ and $V_{\pi'}=Y_2\otimes Y_4\otimes\ldots\otimes Y_{2l}$.
Since $\Pi(M_{2i-1}/M_{2i-2})\subseteq\pi$ and $V_{2i-1}$ is irreducible, then $\pi(Y_{2i-1})\subseteq\pi$, and so $\pi(c_{2i-1})\subseteq\pi$.
Similarly $\pi(c_{2i})\subseteq\pi'$.
Since $\res{c}{J}=c_1c_2\cdots c_{2l}$, we have that $c_1c_3\cdots c_{2l-1}=\res{c_{\pi}}{J}$ and $c_2c_4\cdots c_{2l}=\res{c_{\pi'}}{J}$, and thus $V_{\pi}\in\Irr(J|c_{\pi})$ and $V_{\pi'}\in\Irr(J|c_{\pi'})$.
Since $M_{2i-1}/M_{2i-2}$ is a normal $\pi$-subgroup of $J/M_{2i-2}$, then $M_{2i-1}$ is contained in $M_{2i-2}H$.
Since $V_{2i-1}$ is irreducible, we have that $\res{Y_{2i-1}}{H}$ is irreducible, and so $\res{V_{\pi}}{H}$ is irreducible.
Similarly for $\res{V_{\pi'}}{H'}$.
Finally, since $\Pi(H),\Pi(c_{\pi})\subseteq\pi$ and $\Pi(H'),\Pi(c_{\pi'})\subseteq\pi'$, then $\res{V_{\pi'}}{H}$ and $\res{V_{\pi}}{H'}$ are ordinary modules.
\end{proof}
Finally, we relate this result with the Itô--Michler theorem.
\begin{cor}\label{Satz: pi-decomposition degree}
In the decomposition of $V$ given in \autoref{Satz: pi-decomposition}, we have that $\dim(V)_{\pi}=|G:J|_\pi\cdot\dim(V_{\pi})$ and $\dim(V)_{\pi'}=|G:J|_{\pi'}\cdot\dim(X')$.
In particular, $\Pi(V)\subseteq\pi'$ if and only if $H$ is a Hall $\pi$-subgroup of $G$, $c_\pi=1$, $V_{\pi}$ is a linear module, and $\res{U}{H'}=\res{V_{\pi'}}{H'}$ is irreducible for $U=V_{\pi}\otimes V_{\pi'}$.
\end{cor}

\section*{Acknowledgements}
This research was conducted during a sabbatical visit of the second author to the Department of Mathematics ``Ulisse Dini'' of the Univerity of Florence, Italy. The authors would like to thank the department, in particular Emanuele Pacifici and Lucia Sanus for their advice in writing this paper.

\vspace{5mm}
\noindent
\textbf{Mariagrazia Bianchi}\\
{Università degli Studi di Milano,\\Dipartimento di Matematica Federigo Enriques,\\ Via Saldini 50, 20123, Milano, Italy\\ mariagrazia.bianchi@unimi.it}

\vspace{5mm}
\noindent
\textbf{Nicola Sambonet}\\
{Universidade Federal da Bahia,\\ Instituto de Matemática e Estatística,\\
Avenida Milton Santos s/n, Campus universitário de Ondina,\\ 40170-110, Salvador, Bahia, Brazil}\\
nsambonet@gmail.com
\end{document}